\documentclass[12pt,twoside]{amsart} 
\title{Surfaces of globally $F$-regular and $F$-split type}
\author[Y. Gongyo]{Yoshinori  Gongyo}
\address{Graduate School of Mathematical Sciences, 
the University of Tokyo, 3-8-1 Komaba, Meguro-ku, Tokyo 153-8914, Japan.}
\email{gongyo@ms.u-tokyo.ac.jp}
\address{Department of Mathematics, Imperial College London, 180 Queen's Gate, London SW7 2AZ, UK}
\email{y.gongyo@imperial.ac.uk}
\author[S. Takagi]{Shunsuke  Takagi}
\address{Graduate School of Mathematical Sciences, 
the University of Tokyo, 3-8-1 Komaba, Meguro-ku, Tokyo 153-8914, Japan.}
\email{stakagi@ms.u-tokyo.ac.jp}
\dedicatory{Dedicated to Professor~Robert~Lazarsfeld on the~occasion of his~sixtieth~birthday.}

\newcommand{\Supp}[0]{{\operatorname{Supp}}}

\newcommand{\Spec}[0]{{\operatorname{Spec}}}

\newcommand{\Q}{\mathbb{Q}} 
\newcommand{\Z}{\mathbb{Z}} 

\usepackage{pxfonts}

\usepackage{latexsym}
\usepackage{amsmath}
\usepackage{amssymb}
\usepackage{amsthm}
\usepackage{amscd}
\usepackage{enumerate} 
\usepackage{amssymb} 
\usepackage{mathrsfs}          
\usepackage[all]{xy}
\usepackage{color}
\usepackage{xypic}

\usepackage[%
 colorlinks=true]{hyperref}

\newtheorem{thm}{Theorem}[section]

\newtheorem{prop}[thm]{Proposition}

\newtheorem{lem}[thm]{Lemma}

\newtheorem{cl}[thm]{Claim}
 
\theoremstyle{definition}

\newtheorem{defi}[thm]{Definition}

\newtheorem{eg}[thm]{Example}

\newtheorem{rem}[thm]{Remark}

\newtheorem*{ack}{Acknowledgments}

\begin{document}
\bibliographystyle{amsalpha+}
 
 \maketitle
 
 \begin{abstract}
We prove that normal projective surfaces of dense globally $F$-split type (resp. globally $F$-regular type) are of Calabi--Yau type (resp. Fano type). 
 \end{abstract}


\section{Introduction}

The notion of globally $F$-split varieties was introduced by Mehta and Ramanathan \cite{MeR} in the 1980s to study cohomology of Schubert varieties. 
Later, globally $F$-regular varieties, a special class of globally $F$-split varieties, were introduced by K.~Smith \cite{Sm} who drew inspiration from the tight closure theory. 
Global $F$-splitting and global $F$-regularity both are global properties of a projective variety over a  field of positive characteristic as the name suggests, and they impose strong conditions on the structure of the variety. 
For example, Kodaira-type vanishing theorems, which generally do not hold in positive characteristic, hold on such a variety.  
Those notions have many applications to representation theory and birational geometry in positive characteristic. 

Using reduction to positive characteristic, global $F$-splitting and global $F$-regularity make sense in characteristic zero as well: 
let $X$ be a normal projective variety over an algebraically closed field of characteristic zero. 
$X$ is said to be of globally $F$-regular type (resp. dense globally $F$-split type) if its modulo $p$ reduction is globally $F$-regular for almost all $p$ (resp. infinitely many $p$). 
In this paper, we consider a geometric interpretation of these properties, especially focusing on the surface case. 

We say that $X$ is of Fano type (resp. Calabi--Yau type) if there exists an effective $\Q$-divisor $\Delta$ on $X$ such that $-(K_X+\Delta)$ is ample (resp. $\Q$-linearly trivial) and $(X, \Delta)$ is klt (resp. log canonical). 
Schwede and Smith then asked in \cite[Question 7.1]{schsmith-logfano} whether varieties of globally $F$-regular type (resp. dense globally $F$-split type) are of Fano type (resp. Calabi--Yau type) or not. 
We studied this question in an earlier paper \cite{gost} by the same authors together with Okawa and Sannai, and gave an affirmative answer when $X$ is a Mori dream space. 
However, being a Mori dream space is a very strong condition that is not generally satisfied for varieties of dense globally $F$-split type.  
It is, therefore, natural to ask whether the result still holds without the assumption of being a Mori dream space. 

Curves of globally $F$-regular type are nothing but rational curves.
Smooth curves of globally $F$-split type are elliptic curves besides rational curves.
Thus, the first nontrivial case is when $X$ is a surface, and we settle this case. 
The following is our main result. 

\begin{thm}[=Theorems \ref{F-reg surface} and \ref{F-split surface}]\label{main1}
Let $S$ be a normal projective surface over an algebraically closed field of characteristic zero. 
If $S$ is of dense globally $F$-split type $($resp. globally $F$-regular type$)$, then it is of Calabi--Yau type $($resp. Fano type$)$. 
\end{thm}

One of the key ingredients in the proof is to show that taking the Zariski decomposition of the anti-canonical divisor of a surface of dense globally $F$-split type commutes with reduction modulo $p$. 
The globally $F$-regular case of Theorem \ref{main1} immediately follows from this fact. 

The proof of the globally $F$-split case is much more involved. 
First, by taking the minimal resolution, we may assume that $S$ is smooth. If $S$ is not rational, then the problem can be reduced to whether the projective bundle of a rank $2$ vector bundle of degree zero over an elliptic curve is globally $F$-split.  
This question was already answered by Mehta and Srinivas \cite{MS}, so we suppose that  $S$ is rational. 
Using the Zariski decomposition of $-K_S$ and a result of Laface and Testa \cite{lt} on rational surfaces, we can reduce to the case where $-K_S$ is nef and there exists an effective divisor $D$ linearly equivalent to $-K_S$.  
We can assume in addition that the modulo $p$ reduction $S_p$ of $S$ is a minimal elliptic surface and the reduction $D_p$ of $D$ is an indecomposable curve of canonical type. 
We then make use of the classification of singular fibers (Kodaira's table) to see that if $D_p$ is not of type $I_n$, then $(S_p, D_p)$ has to be globally $F$-split\footnote{The notions of global $F$-splitting and global $F$-regularity can be extended to a pair of a normal projective variety and a divisor. See Definition \ref{defi-gFreg} for the details.} for infinitely many $p$.
Finally, since a fiber of type $I_n$ is a normal crossing divisor and global $F$-splitting implies log canonicity (see \cite[Theorem 3.9]{hw}), we conclude that  $(S, D)$ is log canonical, that is, $S$ is of Calabi-Yau type. 

We remark that the globally $F$-regular case of Theorem \ref{main1} was first proved by Okawa \cite{okawa_phd} using the deformation theory in mixed characteristic, but our proof is more geometric.  
Also, during the preparation of this manuscript, Hwang and Park \cite{hp} announced that they proved the globally $F$-regular case with a different method. 
Their method relies on Sakai's results \cite{sakai2}, \cite{sakai3} and \cite{sakai}. 
Some (not all) results in Section \ref{main section} may follow from his results, but our argument is substantially different from his argument.  
 
 \begin{ack} 
 The authors wish to thank Professors Antonio Laface and Vasudevan Srinivas for answering their questions. They also thank Osamu Fujino, Atsushi Ito and Ching-Jui Lai for careful reading of this manuscript and for helpful comments. 
They are grateful to Shinnosuke Okawa, Akiyoshi Sannai and Taro Sano for valuable conversations. 

The first author was partially supported by the Grand-in-Aid for Research Activity Start-Up $\sharp$24840009 from JSPS and Research expense from the JRF fund. The second author was partially supported by Grant-in-Aid for Young Scientists (B) 23740024 from JSPS. 
This material is based upon work supported by the National Science Foundation under Grant No.~0932078 000, while the second author was in residence at the Mathematical Science Research Institute in Berkeley, California, during the spring semester 2013 of the program Commutative Algebra.  
 \end{ack}

We will freely use the standard notations in \cite{komo}. 

\section{Preliminaries}\label{preliminaries}

\subsection{Varieties of Fano type and of Calabi--Yau type}
In this paper, we use the following terminology.  

\begin{defi}[{cf.~\cite[Definition 2.34]{komo}, \cite[Remark 4.2]{schsmith-logfano}}]\label{sing of pairs}
Let $X$ be a normal variety over a field $k$ of \textit{arbitrary characteristic} and $\Delta$ be an effective $\mathbb{Q}$-divisor on $X$ such that $K_X+\Delta$ is $\mathbb{Q}$-Cartier. 
Let $\pi: \widetilde{X} \to X$ be a birational morphism from a normal variety $\widetilde{X}$.
Then we can write 
$$K_{\widetilde{X}}=\pi^*(K_X+\Delta)+\sum_{E}a(E, X, \Delta) E,$$ 
where $E$ runs through all the distinct prime divisors on $\widetilde{X}$ and the $a(E, X, \Delta)$ are rational numbers. 
We say that the pair $(X, \Delta)$ is \textit{log canonical} (resp. \textit{klt}) if $a(E, X, \Delta) \ge -1$ (resp. $a(E, X, \Delta) >-1$) for every prime divisor $E$ over $X$.  
If $\Delta=0$, we simply say that $X$ has only log canonical singularities (resp. log terminal  singularities). 
\end{defi}

\begin{defi}[cf. {\cite[Lemma-Definition 2.6]{prokshok-mainII}}]\label{Fano pair}Let $X$ be a projective normal variety over a field and $\Delta$ be an effective $\mathbb{Q}$-divisor on $X$ such that $K_X+\Delta$ is $\mathbb{Q}$-Cartier. 
\begin{enumerate}[(i)]
\item We say that $(X,\Delta)$ is {\em log Fano} if $-(K_X+\Delta)$ is ample and $(X, \Delta)$ is klt. 
We say that $X$ is of {\em Fano type} if there exists an effective $\mathbb Q$-divisor $\Delta$ on $X$ such that $(X,\Delta)$ is log Fano. 
\item  We say that $(X,\Delta)$ is {\em log Calabi-Yau} if $K_X+\Delta \sim_{\mathbb{Q}}0$ and $(X, \Delta)$ is log canonical. 
We say that $X$ is of \textit{Calabi--Yau type} if 
there exists an effective $\mathbb{Q}$-divisor $\Delta$ on $X$ such that $(X, \Delta)$ is log Calabi-Yau.  
\end{enumerate}
\end{defi}

\begin{rem}\label{rm-weakfano} 
Let $X$ be a projective variety over an algebraically closed field of characteristic zero.
If there exists an effective $\mathbb{Q}$-divisor $\Delta$ on $X$ such that $(X,\Delta)$ is klt (resp. log canonical) and $-(K_X+\Delta)$ is nef and big (resp. semi-ample), then $X$ is of Fano type (resp. Calabi-Yau type). The reader is referred to \cite[Lemma-Definition 2.6]{prokshok-mainII} for more details. 
\end{rem}

\begin{lem}[cf.~\cite{gost}]\label{images}
Let $f: X \to Y$ be a birational morphism of normal projective varieties over an algebraically closed field.  
If $X$ is of Fano type $($resp. Calabi--Yau type$)$, then so is $Y$. 
\end{lem}
\begin{proof}
The Fano type case is a special case of \cite[Theorem 5.5]{gost}, so we only prove the Calabi--Yau type case. 

Suppose that $X$ is of Calabi--Yau type, that is, there exists an effective $\mathbb{Q}$-divisor $\Delta_X$ on $X$ such that $K_X+\Delta_X \sim_{\mathbb{Q}}0$ and $(X, \Delta_X)$ is log canonical. 
Letting $\Delta_Y:=f_*{\Delta_X}$, one has $K_Y+\Delta_Y=f_*(K_X+\Delta_X) \sim_{\mathbb{Q}} 0$. 
Then $(K_X+\Delta_X)-f^*(K_Y+\Delta_Y)$ is an $f$-exceptional divisor which is $\mathbb{Q}$-linearly trivial relative to $f$. Hence, $K_X+\Delta_X=f^*(K_Y+\Delta_Y)$. 
Since $(X, \Delta_X)$ is log canonical, $(Y, \Delta_Y)$ is also log canonical, which implies that $(Y, \Delta_Y)$ is log Calabi-Yau. 
\end{proof}

\subsection{Globally $F$-regular and $F$-split varieties}

In this subsection, we review the definitions and basic properties of {\em  global $F$-regularity} and {\em global $F$-splitting}. 

A field $k$ of prime characteristic $p$ is called \textit{$F$-finite} if $[k:k^p]<\infty$. 

\begin{defi}[\textup{\cite[Definition 3.1]{schsmith-logfano}}]\label{defi-gFreg} 
Let $X$ be a normal projective variety defined over an $F$-finite field of characteristic $p>0$ and $\Delta$ be an effective $\Q$-divisor on $X$. 
\begin{enumerate}[(i)]
\item 
We say that $(X, \Delta)$ is \textit{globally sharply $F$-split} if there exists $e \in \mathbb{N}$ for which the composition map 
$$\mathcal{O}_X \to F^e_*\mathcal{O}_X \hookrightarrow F^e_*\mathcal{O}_X(\lceil (p^e-1) \Delta \rceil)$$
of the $e$-times iterated Frobenius map $\mathcal{O}_X \to F^e_*\mathcal{O}_X$ with a natural inclusion $F^e_*\mathcal{O}_X \hookrightarrow F^e_*\mathcal{O}_X(\lceil (p^e-1) \Delta \rceil)$
splits as an $\mathcal{O}_X$-module homomorphism. 

\item
We say that $(X, \Delta)$ is \textit{globally $F$-regular} if for every effective divisor $D$ on $X$, there exists $e \in \mathbb{N}$ such that the composition map 
$$\mathcal{O}_X \to F^e_*\mathcal{O}_X \hookrightarrow F^e_*\mathcal{O}_X(\lceil (p^e-1) \Delta \rceil+D)$$
of the $e$-times iterated Frobenius map $\mathcal{O}_X \to F^e_*\mathcal{O}_X$ with a natural inclusion $F^e_*\mathcal{O}_X \hookrightarrow F^e_*\mathcal{O}_X(\lceil (p^e-1) \Delta \rceil+D)$
splits as an $\mathcal{O}_X$-module homomorphism. 
\end{enumerate}
When $\Delta=0$, we simply say that $X$ is globally $F$-split (resp. globally $F$-regular). 
\end{defi}

\begin{lem}\label{F-split lemma} 
Let $X$ be a normal projective variety over an $F$-finite field of characteristic $p>0$.  
\begin{enumerate}[$(1)$]
\item $(${\cite[Lemma 2.14]{gost}}$)$ Let $f: X \dashrightarrow Y$ be a small birational map or an algebraic fiber space to a normal projective variety $Y$. 
If $X$ is globally $F$-regular (resp. globally $F$-split), then so is $Y$.

\item $(${\cite[Theorem 4.3]{schsmith-logfano}}$)$  
If $X$ is globally $F$-regular $($resp. globally $F$-split$)$, then there exists an effective $\Q$-divisor $\Delta$ on $X$ such that $(X, \Delta)$ is globally $F$-regular (resp. globally $F$-split) with $-(K_X+\Delta)$ ample (resp. $\Q$-trivial). 
\end{enumerate}
\end{lem}

Now we briefly explain how to reduce things from characteristic zero to characteristic $p > 0$. 
The reader is referred to \cite[Chapter 2]{HH} and \cite[Section 3.2]{MS} for further details. 

Let $X$ be a normal variety over a field $k$ of characteristic zero and $D=\sum_i d_i D_i$ be a $\mathbb{Q}$-divisor on $X$. 
Choosing a suitable finitely generated $\mathbb{Z}$-subalgebra $A$ of $k$, 
we can construct a scheme $X_A$ of finite type over $A$ and closed subschemes $D_{i, A} \subsetneq X_A$ such that 
there exists isomorphisms 
\[\xymatrix{
X \ar[r]^{\cong \hspace*{3em}} &  X_A \times_{\Spec \, A} \Spec \, k\\
D_i \ar[r]^{\cong \hspace*{3em}} \ar@{^{(}->}[u] & D_{i, A} \times_{\Spec \, A} \Spec \, k. \ar@{^{(}->}[u]\\
}\]
Note that we can enlarge $A$ by localizing at a single nonzero element and replacing $X_A$ and $D_{i,A}$ with the corresponding open subschemes. 
Thus, applying the generic freeness \cite[(2.1.4)]{HH}, we may assume that $X_A$ and $D_{i, A}$ are flat over $\Spec \, A$.
Enlarging $A$ if necessary, we may also assume that $X_A$ is normal and $D_{i, A}$ is a prime divisor on $X_A$. 
Letting $D_A:=\sum_i d_i D_{i,A}$, we refer to $(X_A, D_A)$ as a \textit{model} of $(X, D)$ over $A$.   
Given a closed point $\mu \in \Spec \, A$, we denote by $X_{\mu}$ (resp., $D_{i, \mu}$) the fiber of $X_A$ (resp., $D_{i, A}$) over $\mu$.  
Then $X_{\mu}$ is a scheme of finite type over the residue field $k(\mu)$ of $\mu$, which is a finite field.  
Enlarging $A$ if necessary, we may assume that  $X_{\mu}$ is a normal variety over $k(\mu)$, $D_{i, \mu}$ is a prime divisor on $X_{\mu}$ and consequently $D_{\mu}:=\sum_i d_i D_{i, \mu}$ is a $\mathbb{Q}$-divisor on $X_{\mu}$ for all closed points $\mu \in \Spec \, A$. 

Given a morphism $f:X \to Y$ of varieties over $k$ and a model $(X_A, Y_A)$ of $(X, Y)$ over $A$,  after possibly enlarging $A$, we may assume that $f$ is induced by a morphism $f_A :X_A \to Y_A$ of schemes of finite type over $A$. 
Given a closed point $\mu \in \Spec \, A$, we obtain a corresponding morphism $f_{\mu}:X_{\mu} \to Y_{\mu}$ of schemes of finite type over $k(\mu)$. 
If $f$ is projective (resp. finite), after possibly enlarging $A$, we may assume that $f_{\mu}$ is projective (resp. finite) for all closed points $\mu \in \Spec \, A$. 

We denote by $X_{\bar{\mu}}$ the base change of $X_{\mu}$ to the algebraic closure $\overline{k(\mu)}$ of $k(\mu)$. 
Similarly for $D_{\bar{\mu}}$ and $f_{\bar{\mu}}:X_{\bar{\mu}} \to Y_{\bar{\mu}}$.
Note that $(X_{\mu}, D_{\mu})$ is globally $F$-regular (resp. globally sharply $F$-split)  if and only if so is $(X_{\bar{\mu}}, D_{\bar{\mu}})$. 

\begin{defi}\label{type}
Let the notation be as above. Suppose that $X$ is a normal projective variety over a field of characteristic zero and $\Delta$ is an effective $\Q$-divisor on $X$. 
\begin{enumerate}[(i)]
\item $(X, \Delta)$ is said to be of \textit{dense globally sharply $F$-split type} if for a model of $(X, \Delta)$ over a finitely generated $\mathbb{Z}$-subalgebra $A$ of $k$, there exists a dense subset of closed points $W \subseteq \Spec \, A$ such that $(X_{\mu}, \Delta_{\mu})$ is globally sharply $F$-split for all $\mu \in W$. 
\item $(X, \Delta)$ is said to be of \textit{globally $F$-regular type} if for a model of $(X, \Delta)$ over a finitely generated $\mathbb{Z}$-subalgebra $A$ of $k$, there exists a dense open subset of closed points $W \subseteq \Spec \, A$ such that $(X_{\mu}, \Delta_{\mu})$ is globally $F$-regular for all $\mu \in W$. 
\end{enumerate}
This definition is independent of the choice of a model. 
When $\Delta=0$, we simply say that $X$ is of dense globally $F$-split type (resp. globally $F$-regular type). 
\end{defi}

\begin{prop}\label{just singularities}
Let $X$ be a normal projective variety over a field of characteristic zero and $\Delta$ be an effective $\Q$-divisor on $X$ such that $K_X+\Delta$ is $\Q$-Cartier. 
\begin{enumerate}[$(1)$]
\item $($\cite[Theorem 3.9]{hw}$)$ If $(X, \Delta)$ is of globally $F$-regular type $($resp. dense globally $F$-split type$)$, then it is klt $($resp. log canonical$)$. 
\item $($\cite[Theorem 5.1]{schsmith-logfano}$)$ If $X$ is of Fano type, then $X$ is of globally $F$-regular type.  
\end{enumerate}
\end{prop}


\section{Frobenius splitting of projective bundles}

\begin{prop}\label{trivial vec}
Suppose that the pair $(X,\Delta)$ is globally $F$-regular (resp. globally sharply $F$-split). 
Let $\mathcal{L}_i$ be a line bundle on $X$ for  $i=1, \dots, r$, $V=\bigoplus_{i=1}^{r}\mathcal{L}_i$ and $\pi:\mathbb{P}_X(V) \to X$ be the projective bundle associated to $V$. 
 Then $(\mathbb{P}_X(V), \pi^{[*]}\Delta)$ is also globally $F$-regular (resp. globally sharply $F$-split), where  $\pi^{[*]}$ denotes the pullback by the flat morphism $\pi$.
 \end{prop} 

\begin{proof}
We only prove the globally $F$-regular case. The proof of the globally sharply $F$-split case is similar (and simpler).  

Take an effective ample divisor $D$ on $X$ such that  $D \geq \Delta$ and 
 $$\mathbb{P}_X(V) \setminus \pi^*D \cong (X \setminus D) \times \mathbb{P}^{r-1}.$$
In order to prove that $(\mathbb{P}_X(V), \pi^{[*]}\Delta)$ is globally $F$-regular, by \cite[Theorem 3.9]{schsmith-logfano}, it suffices to show that there exists $e \in \mathbb{N}$ such that the natural map 
$$\mathcal{O}_{\mathbb{P}_X(V)} \to F^e_*\mathcal{O}_{\mathbb{P}_X(V)}(\lceil (p^e-1)\pi^{[*]}\Delta \rceil+ \pi^*D)$$ 
splits. 
Since $(X, \Delta)$ is globally $F$-regular, there exist $e \in \mathbb{N}$ and an $\mathcal{O}_X$-module homomorphism 
 $$\phi:F^e_*\mathcal{O}_X(\lceil (p^e-1)\Delta \rceil +D) \to \mathcal{O}_X$$
 that sends $1$ to $1$. 
 For each $m_1, \dots, m_r \ge \mathbb{N}$,  if they are all divisible by $p^e$, then by tensoring $\phi$ with $L_1^{m_1/p^e} \otimes \cdots \otimes L_r^{m_r/p^e}$,  we have a map 
 \begin{align*}
 \phi_{m_1, \dots, m_r}: F^e_*(L_1^{m_1} \otimes \cdots \otimes L_r^{m_r}(\lceil (p^e-1)\Delta \rceil+D)) & \to  L_1^{m_1/p^e} \otimes \cdots \otimes L_r^{m_r/p^e} \\
& \hookrightarrow  \bigoplus_{m \ge 0}S^m V.
 \end{align*}
Otherwise, we define the map 
$$\phi_{m_1, \dots, m_r}: F^e_*(L_1^{m_1} \otimes \cdots \otimes L_r^{m_r}(\lceil (p^e-1)\Delta \rceil+D)) \to \bigoplus_{m \ge 0}S^m V$$ to be zero. 
Putting the $\phi_{m_1, \dots, m_r}$ together, we obtain a map 
$$F^e_*\bigoplus_{m \ge 0}S^m V(\lceil (p^e-1)\Delta \rceil+D) \to \bigoplus_{m \ge 0}S^m V$$
which sends $1$ to $1$. 
This induces a map 
$$F^e_*\mathcal{O}_{\mathbb{P}_X(V)}(\pi^{[*]} \lceil (p^e-1)\Delta \rceil+\pi^*D) \to \mathcal{O}_{\mathbb{P}_X(V)}$$
which sends $1$ to $1$.  
Note that $\pi^{[*]}\lceil (p^e-1)\Delta\rceil= \lceil (p^e-1)\pi^{[*]}  \Delta \rceil$, because $X$ is normal  and $\pi$ is smooth in codimension one.
Therefore, $\mathbb{P}_X(V)$ is globally $F$-regular.
\end{proof}

If the vector bundle over a globally $F$-split variety does not split, then its projective bundle is not  globally $F$-split in general. 

\begin{eg}[{See \cite[Remark 1]{MS}}]\label{elliptic ruled}
Let $E$ be an elliptic curve over $\overline{\mathbb{F}}_p$ and $V$ be an indecomposable rank $2$ vector bundle with trivial determinant over $E$. 
Then $\mathbb{P}_Z(E)$ is not globally $F$-split but of log Calabi-Yau type  at lest when $p \geq 5$.  
\end{eg}

We use the following lemma in the proof of the main results. 
This is a generalization of \cite[Proposition 1.4]{hawayo-rees} to the log setting, and the proof is essentially the same as that of \textit{loc. cit}. 
\begin{lem}\label{lem_crepant}
Let $f:X \to Y$ be a birational morphism of normal projective varieties over an $F$-finite field of characteristic $p>0$. 
Suppose that there exist an effective $\Q$-divisor $\Delta_X$ on $X$ and an effective $\Q$-divisor $\Delta_Y$ on $Y$ such that $K_X+\Delta_X$ and $K_Y+\Delta_Y$ are $\Q$-Cartier and that $f^*(K_Y+\Delta_Y)=K_X+\Delta_X$. 
Then $(X,\Delta_X)$ is globally $F$-regular (resp. globally sharply $F$-split) if and only if so is $(Y,\Delta_Y)$.
\end{lem}

\begin{proof}
We only prove the globally $F$-regular case. The proof of the globally sharply $F$-split case is similar (and simpler).  

Take an ample Cartier divisor $H_Y$ on $Y$ and an ample divisor $H_X$ on $X$ such that $f^*H_Y \geq H_X$. 
In order to prove that $(X,\Delta_X)$ is globally $F$-regular, it suffices to show that for any integer $n \ge 0$ and nonzero element $c \in H^0(X, nH_X)$, there exists $e \in \mathbb{N}$ such that 
the map 
$$\mathcal{O}_X \to F^e_*\mathcal{O}_X(\lceil (p^e-1) \Delta_X \rceil) \xrightarrow{ \times F^e_*c} F^e_*\mathcal{O}_X(\lceil (p^e-1) \Delta_X \rceil+nH_X)$$
splits. 

Fix  any $n \ge 0$ and nonzero element $c \in H^0(X, nH_X)$. 
Since $f^*H_Y \geq H_X$, one has that $f_*c \in H^0(Y, nH_Y)$.
It then follows from the global $F$-regularity of $(Y,\Delta_Y)$ that there exist $e \in \mathbb{N}$ and an $\mathcal{O}_Y$-module homomorphism 
$$\phi: F^e_*\mathcal{O}_Y(\lceil (q-1)\Delta_Y \rceil +nH_Y) \to \mathcal{O}_Y$$
which sends $F^e_*f_*c$ to $1$. 
 By Grothendieck duality and the fact that $f^*H_Y \geq H_X$,  we obtain the following composition map: 
\begin{align*}
& \mathrm{Hom}_{\mathcal{O}_Y}(F^e_*\mathcal{O}_Y(\lceil (q-1)\Delta_Y \rceil +nH_Y), \mathcal{O}_{Y})\\
 \cong & H^0(Y, -\lceil (q-1)(K_Y+\Delta_Y) \rceil -nH_Y)\\
\subseteq & H^0(X, -\lceil (q-1)(K_X+\Delta_X) \rceil -nH_X)\\
\cong & \mathrm{Hom}_{\mathcal{O}_X}(F^e_*\mathcal{O}_X(\lceil (q-1)\Delta_X \rceil +nH_X), \mathcal{O}_{X}). 
\end{align*} 
Hence, $\phi$ induces a map 
$$f^*\phi: F^e_*\mathcal{O}_X(\lceil (q-1)\Delta_X \rceil +nH_X) \to \mathcal{O}_{X}$$ 
which sends $F^e_*c$ to $1$. 
We have seen above that this implies the global $F$-regularity of $(X, \Delta_X)$. 

The converse argument just reverses this. The lemma is proved.
\end{proof}

\begin{rem}[\textup{\cite{MS}, \cite[Lemma 2.1]{hawayo-rees}}]\label{min resolution}
Since globally $F$-split surfaces are $\mathbb{Q}$-Gorenstein by \cite{MS}, it follows from \cite[Proposition 1.4]{hawayo-rees} (or Lemma \ref{lem_crepant}) that the minimal resolution of a globally $F$-regular (resp. globally $F$-split) surface is globally $F$-regular (resp. globally $F$-split).  
\end{rem}


\section{Zariski decomposition and reduction modulo $p$} 
 
In this section, we study how Zariski decompositions behave under taking reduction modulo $p$.  

\begin{lem}\label{pseudo-effective-d}
Let $X$ be a normal projective variety over an algebraically closed field $k$ of characteristic zero and $D$ be an $\mathbb{R}$-Cartier divisor on $X$. 
Suppose we are  given a model of $(X, D)$ over a finitely generated $\Z$-subalgebra $A$ of $k$. 
If $D$ is not nef (resp. pseudo-effective), then $D_{\mu}$ is not nef (resp. pseudo-effective) for general closed points $\mu \in \Spec \, A$. 
\end{lem}

\begin{proof}
In the pseudo-effective case, we make use of the characterization of pseudo-effective divisors given in \cite{bdpp}. 
By taking a resolution, we may assume that $X$ is smooth. 
Since $D$ is not nef (resp. pseudo-effective), we can find a curve (resp. a covering curve) $C$ such that $C.D<0$. 
Then $C_{\mu}$ is also a curve (resp. a covering curve) and $C_{\mu}.D_{\mu}=C.D<0$ for general closed points $\mu \in \Spec \, A$. 
Thus, we conclude that $D_{\mu}$ is not nef (resp. pseudo-effective) for general closed points $\mu \in \Spec \, A$. 
\end{proof}

\begin{lem}[cf.{\cite{fujita}}]\label{lem_zariski_nu}
Let $S$ be a smooth projective surface over an algebraic closed field $k$ of characteristic zero and $D$ be a pseudo-effective $\mathbb{Q}$-divisor on $S$. 
Let $D=P(D)+N(D)$ be the Zariski decomposition of $D$. 
Suppose we are given a model of $(S, P(D), N(D))$ over a finitely generated $\Z$-subalgebra $A$ of $k$. 
For general closed points $\mu \in \Spec \, A$,  if $D_{\bar{\mu}}$ is pseudo-effective, then $N(D)_{\bar{\mu}} \leq N(D_{\bar{\mu}})$. 
 \end{lem}

\begin{proof}
Let $N(D)=\sum n_iC_i$ be the decomposition of $N(D)$ into prime divisors. 
We can assume that $D_{\bar{\mu}}$ is pseudo-effective, $N(D)_{\bar{\mu}}$ has a negative definite intersection matrix, and $P(D)_{\bar{\mu}}. C_{i, \bar{\mu}}=P(D). C_i=0$ for all $i$. 
If $P(D)_{\bar{\mu}}$ is nef, then $D_{\bar{\mu}}=P(D)_{\bar{\mu}}+N(D)_{\bar{\mu}}$ is the Zariski decomposition of $D_{\bar{\mu}}$. 
Then the assertion holds by the uniqueness of the Zariski decomposition (\cite[(1.12) Theorem]{fujita}).  
Thus, we may assume that $P(D)_{\bar{\mu}}$ is not nef.  

Note by \cite[(1.8) Lemma]{fujita} that $P(D)_{\bar{\mu}}$ is pseudo-effective, because  $P(D)_{\bar{\mu}}. C_{i, \bar{\mu}}=0$ for all $i$ and $N(D)_{\bar{\mu}}$ has a negative definite intersection matrix.  
It then follows from \cite[(1.12) Theorem]{fujita} that $P(D)_{\bar{\mu}}$ has the Zariski decomposition $P(D)_{\bar{\mu}}=P+N'$.
Let $N=N(D)_{\bar{\mu}}+N'$ and $N'=\sum n'_jC'_j$ be the decomposition of $N'$ into prime divisors.  
It is enough to show that $D_{\bar{\mu}}=P+N$ is the Zariski decomposition of $D_{\bar{\mu}}$. 
Since $P$ is nef and $P(D)_{\bar{\mu}}.C_{i, \bar{\mu}}=0$ for all $i$, 
we see that if $P.C_{i, \bar{\mu}} \ne 0$ for some $i$, then $N'.C_{i, \bar{\mu}}<0$, which implies that $C_{i, \bar{\mu}}$ is contained in the support of $N'$. 
This contradicts the fact that $P(D)_{\bar{\mu}}=P+N'$ is the Zariski decomposition, so $P.C_{i, \bar{\mu}}=0$ and $C_{i, \bar{\mu}}.C'_j=0$ for all $i,j$. 
Then $N$ has a negative definite intersection matrix, and we can conclude by the uniqueness of the Zariski decomposition again that $D_{\bar{\mu}}=P+N$ is the Zariski decomposition of $D_{\bar{\mu}}$. 
\end{proof}

If $S$ is a smooth projective surface of dense globally $F$-split type, then $-K_S$ is pseudo-effective by Lemma \ref{pseudo-effective-d}, so we can consider its Zariski decomposition. 
The following proposition plays a key role in this paper.
 \begin{prop}\label{nef-ness}
 Let $S$ be a smooth projective surface of dense globally $F$-split type $($resp. globally $F$-regular type$)$ over an algebraically closed field $k$ of characteristic zero. 
 Let $-K_S=P+N$ be the Zariski decomposition of $-K_S$. 
 Then $(S, N)$ is of dense globally sharply $F$-split type $($resp. globally $F$-regular type$)$, so in particular, $(S, N)$ is log canonical $($resp. klt$)$ by Proposition \ref{just singularities} (1). 
Moreover, given a model of $(S, P, N)$ over a finitely generated $\Z$-subalgebra $A$ of $k$,  
the pair $(S_{\mu}, N_{\mu})$ is globally sharply $F$-split $($resp. globally $F$-regular$)$ and 
$$-K_{S_{\bar{\mu}}}=P_{\bar{\mu}}+N_{\bar{\mu}}$$ 
is the Zariski decomposition of $-K_{S_{\bar{\mu}}}$ for a dense set of closed points $($resp. general closed points$)$ $\mu \in \Spec \, A$. 
\end{prop}

\begin{proof}
By definition, there exists a dense subset (resp. a dense open subset) of closed points $W \subseteq \Spec \, A$ such that $S_{\mu}$ is globally $F$-split (resp. globally $F$-regular) for all $\mu \in W$. 
Note that global $F$-splitting (resp. global $F$-regularity) is preserved under the base field extension $\overline{\mathbb{F}_p}/\mathbb{F}_p$. 
It follows from Lemma \ref{F-split lemma} (2) that there exists an effective $\mathbb{Q}$-divisor $\Delta$ on $S_{\bar{\mu}}$ such that $(S_{\bar{\mu}},\Delta)$ is globally sharply $F$-split (resp. globally $F$-regular) and $K_{S_{\bar{\mu}}} + \Delta \sim_{\mathbb{Q}}0$. 
By a property of the Zariski decomposition and Lemma \ref{lem_zariski_nu}, one can see that 
$$\Delta \geq N(-K_{S_{\bar{\mu}}}) \geq N_{\bar{\mu}}.$$ 
Thus, $(S_{\bar{\mu}}, N_{\bar{\mu}})$ is globally sharply $F$-split (resp. globally $F$-regular) for all $\mu \in W$. This means that $(S, N)$ is of dense globally sharply $F$-split type (resp. globally $F$-regular type). 

To prove the latter assertion, we will show that $P_{\bar{\mu}}$ is nef for all $\mu \in W$. 
It follows from Lemma \ref{non-van} that there exists an effective $\Q$-divisor $D=\sum_i d_i D_i$ on $S$ such that $P \sim_{\mathbb{Q}} D$. 
If $P_{\bar{\mu}}$ is not nef, then there exists some curve $C$ on $S_{\mu}$ such that $P_{\bar{\mu}}.C < 0$. Then $C$ is contained in $\mathrm{Supp}\,D_{\bar{\mu}} $, 
that is, there exists $i$ such that $C=D_{i, \bar{\mu}}$. 
On the other hand, 
$$P_{\mu}.C=P_{\mu}.D_{i, \mu}= P.D_{i}\geq0.$$ 
This is a contradiction. 
Therefore, $P_{\bar{\mu}}$ is nef for all $\mu \in W$. 
Note that $P_{\bar{\mu}}. N_{j, \bar{\mu}}=P. N_j=0$ for every irreducible component $N_j$ of $N$ and that $N_{\bar{\mu}}$ is zero or has a negative definite intersection matrix. 
By the uniqueness of the Zariski decomposition (see \cite[(1.12) Theorem]{fujita}), $-K_{S_{\bar{\mu}}}=P_{\bar{\mu}}+N_{\bar{\mu}}$ is the Zariski decomposition of $-K_{S_{\bar{\mu}}}$ for all $\mu \in W$. 
\end{proof}

In the proof of Proposition \ref{nef-ness},  we used the following lemma. 
 \begin{lem}[{\cite[2.6, Remark-Corollary]{shocompl}, \cite{fk}}]\label{non-van}Let $(S, \Delta)$ be a log canonical surface such that $-(K_S+\Delta)$ is nef. Then $\kappa(-(K_S+\Delta)) \geq 0$. 
\end{lem}
\begin{proof}We refer the reader to \cite[2.6, Remark-Corollary, p.3890]{shocompl} for the proof. 
We remark that this lemma is an easy consequence of the Riemann--Roch formula when $S$ is rational. 
\end{proof}
 
We use the notion of divisors of insufficient fiber type and Lemma \ref{lem-insufficient} in the proof of Theorem \ref{F-split surface}. 
Lemma \ref{lem-insufficient} is well-known to experts, but we include its proof here for the reader's convenience.

\begin{defi}[{\cite[Section 5.a in Chapter III]{nakayama-zariski-abun}}]\label{insufficient}
Let $f:S \to C$ be a projective surjective morphism from a smooth projective surface $S$ to a smooth projective curve $C$ with connected fibers over an algebraic closed field (of any characteristic). 
Let $D$ be an effective $f$-vertical $\mathbb{Q}$-divisor on $S$.  
We say that $D$ is of {\em insufficient fiber type} over $C$ if for any closed point $x$ of $\mathrm{Supp}\,f_*D$,  there exists a prime divisor $\Gamma$ on $S$ such that $f(\Gamma) = x$ and $\Gamma \not \subset \Supp\,D$.  
\end{defi}

\begin{lem}[{\cite[5.3 Corollary and 5.7 Proposition in Chapter III]{nakayama-zariski-abun}}]\label{lem-insufficient}
Suppose that $f:S \to C$ and $D$ be the same as in Definition \ref{insufficient}. 
If $D$ is of insufficient fiber type over $C$, then $D$ is not nef and $D=N(D)$. 
\end{lem}
\begin{proof}
Suppose to the contrary that $D$ is nef. 
Since $D$ is of insufficient fiber type over $C$,  taking into account that all fibers of $f$ are connected, we can find a curve $\Gamma$ such that $f(\Gamma) \subseteq \mathrm{Supp}\,f_*D$ and $D.\Gamma >0$. 
Let $F=f^*f(\Gamma)$. Then $D.F=0$ and $F \geq \Gamma$, but this contradicts the nefness of $D$.  
Thus, every divisor of insufficient fiber type is not nef. 
Let $D=P+N$ be the Zariski decomposition of $D$. 
Then $P$ is also of insufficient fiber type, unless $P$ is zero. 
Since $P$ is nef, $P$ has to be zero, that is, $D=N$.
\end{proof}


\section{Main results}\label{main section} 
In this section, we prove Theorem \ref{main1}. 
First we show the globally $F$-regular case. 

\begin{thm}[cf.~\cite{okawa_phd}]\label{F-reg surface}
Let $S$ be a normal projective surface over an algebraically closed field $k$ of characteristic zero. 
Then $S$ is of globally F-regular type if and only if $S$ is of Fano type.
\end{thm}

\begin{proof}[Proof of Theorem \ref{F-reg surface}] 
The if part follows from Proposition \ref{just singularities} (2), so we will prove the only if part. 
Taking the minimal resolution, by Lemma \ref{images} and Remark \ref{min resolution}, we may assume that $S$ is smooth. 
Let $-K_S=P+N$ be the Zariski decomposition of $-K_S$. 
We take a model of $(S, P, N)$ over a finitely generated $\Z$-subalgebra $A$ of $k$ such that
$S_{\mu}$ is globally $F$-regular for all closed points $\mu \in \Spec \, A$. 
Denote by $S_{\bar{\mu}}$ the base change of $S_{\mu}$ to the algebraic closure of $k(\mu)$. 
Similarly for $P_{\bar{\mu}}$ and $N_{\bar{\mu}}$. 
It then follows from Proposition \ref{nef-ness} that  $(S, N)$ is klt and $-K_{S_{\bar{\mu}}}=P_{\bar{\mu}}+N_{\bar{\mu}}$ is the Zariski decomposition of $-K_{S_{\bar{\mu}}}$ for general $\mu$. 
Since $-K_{S_{\bar{\mu}}}$ is big, $P_{\bar{\mu}}$ is also big. 
Then $P^2=P_{\bar{\mu}}^2>0$, so $P$ is nef and big. 
Thus, $(S, N)$ is a klt weak log del Pezzo surface, which means that $S$ is of Fano type.
\end{proof}

\begin{rem}
Theorem \ref{F-reg surface} was first proved by Okawa in his Ph.D thesis \cite{okawa_phd}.
His proof depends on the deformation theory in mixed characteristic and the minimal model theory for surfaces in positive characteristic. 
Our proof is more geometric, just a direct application of Proposition \ref{nef-ness}, without using the deformation theory and the minimal model theory in positive characteristic. 
During the preparation of this manuscript, Hwang and Park \cite{hp} announced that they proved Theorem \ref{F-reg surface} using results of Sakai \cite{sakai} (see also \cite{sakai2}, \cite{sakai3}). 
Some (not all) results in Section \ref{main section} may follow from Sakai's results, but our argument is substantially different from his argument.  
\end{rem}

Next we prove the globally $F$-split case. 

\begin{thm}\label{F-split surface}Let $S$ be a normal projective surface of dense globally $F$-split type over an algebraically closed field of characteristic zero. Then $S$ is of Calabi-Yau type.
\end{thm}

\begin{rem}
The converse of Theorem \ref{F-split surface} is open. 
It is known by \cite{FT} that if $S$ is a klt projective surface over an algebraically closed field of characteristic zero such that $K_S \sim_{\Q} 0$, then it is of dense globally $F$-split type. 
\end{rem}

Before giving the proof of Theorem \ref{F-split surface}, we state a couple of lemmas that will be used in the proof. 

\begin{lem}[{\cite[2.6, Remark-Corollary]{shocompl}, \cite{fk}}]\label{non-comp}
Let $(S, \Delta)$ be a projective dlt surface over an algebraically closed field of characteristic zero such that $-(K_S+\Delta)$ is nef and $S$ is a non-rational smooth surface. 
If $S$ is not of Calabi--Yau type, then there exist a birational morphism $\phi:S \to Z$ and an indecomposable rank $2$ vector bundle $V$ with trivial determinant over an elliptic curve $E$ such that 
$Z \simeq \mathbb{P}_E(V)$ and $K_S+\Delta \sim_{\mathbb{Q}, Z} 0$. 
\end{lem}
\begin{proof}
We refer to \cite[2.6, Remark-Corollary, p.3890]{shocompl} for the proof.
\end{proof}

\begin{lem}\label{pic0}
Let $F$ be an indecomposable curve of canonical type over an algebraically closed field $k$ of positive characteristic. 
If $F$ is not of type $I_n$ where $n \ge 0$, then 
$$\mathrm{Pic}^0\,F \simeq \mathbb{G}_a(k), $$
where $\mathbb{G}_a(k)$ is the additive group of $k$.  
See \cite[Section 2]{hl} or \cite[Definition in Section 2]{mum-enriques} for the definition of indecomposable curves of canonical type. 
 \end{lem}
 
 \begin{proof}
We refer to \cite[Proposition 5.2]{hl} for the proof. 
\end{proof}

Now we start the proof of Theorem \ref{F-split surface}.

\begin{proof}[Proof of Theorem \ref{F-split surface}] 
Taking the minimal resolution, by Lemma \ref{images} and Remark \ref{min resolution}, we may assume that $X$ is smooth. 
Let $-K_S=P+N$ be the Zariski decomposition of $-K_S$. 
It then follows from Proposition \ref{nef-ness} that $(S, N)$ is a log canonical pair of dense globally sharply $F$-split type. 
Taking a smooth dlt blow-up of $(S, N)$ (see also \cite[Proposition--Definition 3.1.1]{prok-compl}), by Lemmas \ref{images} and \ref{lem_crepant}, we can assume that there exists an effective $\Q$-divisor $\Delta$ on $S$ such that $(S, \Delta)$ is dlt with $-(K_S+\Delta)$ nef ($S$ is still smooth).   

Suppose to the contrary that $S$ is not of Calabi--Yau type. 
Note that $P$ is not semi-ample. We will use this fact repeatedly in the rest of the proof. 
If $S$ is not rational, then we obtain a smooth projective surface $Z$ as in Theorem \ref{non-comp},  which is of globally $F$-split type by Lemma \ref{F-split lemma} (1). 
This, however, contradicts Example \ref{elliptic ruled}, and thus, $S$ is a smooth rational surface. 

\begin{cl}\label{claim1}
$\kappa(-K_S)=\kappa(-(K_S+N))=0$ and $\nu(-(K_S+N))=1$, where 
$\nu(L)$ is the numerical dimension of a nef Cartier divisor $L$,  which is defined to be the largest natural number $\nu$ such that the cycle $L^{\nu}$ is numerically nontrivial $($cf. \cite[Definition 6-1-1]{kamama}$)$. 
\end{cl}

\begin{proof}[Proof of Claim \ref{claim1}] 
Since $P$ is not semi-ample, by \cite{big}, $-K_X$ is not big. 
By Lemma \ref{non-van}, one has $$1 \ge \nu(-(K_S+N))  \ge \kappa(-(K_S+N)) \geq 0.$$ 
If $-(K_S+N)$ were abundant, that is, $\kappa(-(K_S+N))=\nu(-(K_S+N))$, then $P$ would be semi-ample by \cite[Theorems 4.19 and 4.20]{fg3} (see also \cite[Corollary 1]{MR}). 
Thus, $\kappa(-(K_S+N))=0$ and $\nu(-(K_S+N))=1$. 
Since $P$ is the positive part of the Zariski decomposition of $-K_S$, 
we also see that $\kappa(-K_S)=0$. 

\end{proof}

It follows from Theorem \cite[Theorem 4.1]{lt} that $H^0(S, -K_S) \ne 0$, because $P$ is not semi-ample. 
We pick an effective divisor $D$ on $S$ which is linearly equivalent to $-K_S$. 
Since $H^1(S, -D)=H^1(S, \mathcal{O}_S)=0$, $\Supp\,D$ is connected. 
Take an integer $l \ge 1$ such that $lP$ is Cartier.  

Now we consider reduction from characteristic zero to positive characteristic. 
Given a model of $(S, D, P, N)$ over a finitely generated $\Z$-subalgebra $A$ of $k$, 
there exists a dense set of closed points $W \subseteq \Spec \, A$ such that 
$S_{\mu}$ is globally $F$-split for all $\mu \in W$. 
Fix any $\mu \in W$. 
We denote by $S_{\bar{\mu}}$ the base change of $S_{\mu}$ to the algebraic closure of $k(\mu)$. 
Similarly for $D_{\bar{\mu}}, P_{\bar{\mu}}$ and $N_{\bar{\mu}}$. 
It follows from Proposition \ref{nef-ness} that $-K_{S_{\bar{\mu}}}=P_{\bar{\mu}}+N_{\bar{\mu}}$ is the Zariski decomposition of $-K_{S_{\bar{\mu}}}$. 
Since $\mathcal{O}_{D_{\bar{\mu}}}(l P_{\bar{\mu}})$ is torsion in $\mathrm{Pic} \, D_{\bar{\mu}}$ by \cite[Lemma 2.16]{keel}, 
we see that $P_{\bar{\mu}}$ is semi-ample by \cite[Lemma 5.3]{lt}. 
Let $f:S_{\bar{\mu}} \to \mathbb{P}^1$ be the fibration induced by $P_{\bar{\mu}}$. 
By running a minimal model program over $\mathbb{P}^1$, we obtain a relative minimal elliptic fibration $f':S' \to \mathbb{P}^1$ and a composition of blow-ups of points $\varphi:S_{\bar{\mu}} \to S'$. 
Note that $S'$ is globally $F$-split by Lemma \ref{F-split lemma} (1).  

\begin{cl}\label{claim2} 
There exists a point $x$ on $\mathbb{P}^1$ such that $\varphi_*D_{\bar{\mu}}=\frac{1}{n}f'^*x$, where $n$ is the multiplicity of the fiber $f'^*x$.  
\end{cl}
\begin{proof}[Proof of Claim \ref{claim2}]  
First, we observe that 
$$\varphi_*D_{\bar{\mu}} . f'^*x=-K_{S'}.f'^*x=0$$ for all $x \in \mathbb{P}^1$, which implies that the effective divisor $\varphi_*D_{\bar{\mu}}$ is $f'$-vertical. 
Since $\Supp\, \varphi_*D_{\bar{\mu}}$ is connected, there exists a point $x$ on $\mathbb{P}^1$ such that $\Supp\,\varphi_*D_{\bar{\mu}} \subseteq \Supp \,  f'^*x$. 
Taking into account the fact that $(\varphi_*D_{\bar{\mu}})^2=\varphi_*D_{\bar{\mu}} . K_{S'}=0$, by \cite[Corollary 2.6]{Ba}, we can write that $\varphi_*D_{\mu}=m  F$, where $m \in \mathbb{N}$, $F=\frac{1}{n}f'^*x$ and $n$ is the multiplicity of the fiber $f'^*x$.  

To complete the proof, it remains to show that $m=1$. 
Since $S'$ is a rational surface with $-K_{S'}\sim_{\mathbb{Z}} mF$ nef, by running a minimal model program, we obtain a birational contraction from $S'$ to $\mathbb{P}^2$ or $\mathbb{P}^1 \times \mathbb{P}^1$ which is a composition of blow-ups of points. 
Note that $S'$ is isomorphic neither to $\mathbb{P}^2$ nor to $\mathbb{P}^1 \times \mathbb{P}^1$, because $-K_{S'} \sim_{\mathbb{Z}} mF$ is not big. 
Thus, we have a birational morphism $g : S' \to \mathbb{F}_1$, where $\mathbb{F}_1$ is the Hirzebruch surface of degree $1$. 
Since $mg_*F \sim_{\mathbb{Z}} -K_{ \mathbb{F}_1}=2 S_{\mathbb{F}_1}+3F_{\mathbb{F}_1}$, where  $S_{\mathbb{F}_1}$ is the section and $F_{\mathbb{F}_1}$ is a fiber of $\mathbb{F}_1$, 
we conclude that $m$ has to be equal to one. 
\end{proof} 

Since $\varphi:S_{\bar{\mu}} \to S'$ is a composition of blow-ups of points, 
$$K_{S_{\bar{\mu}}}=\varphi^*K_{S'}+\sum E_i, \quad D_{\bar{\mu}}=\varphi^*F-\sum E_i,$$
where the $E_i$ are the exceptional divisors of $\varphi$ and $F=\varphi_*D_{\bar{\mu}}$ as in the proof of Claim \ref{claim2}. 
Note that $\varphi(E_i) \in \mathrm{Supp}\,F$. 
If there exists some $i$ such that  $ E_{i} \not \subseteq \mathrm{Supp}\,D_{\bar{\mu}}$,  
then by definition, $D_{\bar{\mu}}$ is of insufficient fiber type. 
It then follows from Lemma \ref{lem-insufficient} that $D_{\bar{\mu}}=N_{\bar{\mu}}$, that is, $P=0$. This is a contradiction.
Thus,  all $E_i$ are contained in $\mathrm{Supp}\,D_{\mu}$, so they are components of the reduction $D_{\mu}$ of $D$. 
This means that  there exists a composition of blow-ups of points $\psi: S \to T$ such that $\psi_{\bar{\mu}}\simeq \varphi$ and $T_{\bar{\mu}}\simeq S'$. 

It follows from Lemma \ref{pseudo-effective-d} that $-K_{T}$ is nef. 
We will show that $-K_T$ is not semi-ample. 
Suppose that $-K_T$ is semi-ample, and then let $\tilde{f'}:T \to \mathbb{P}^1$ be the fibration induced by $-K_T$. 
Since $\tilde{f'}_{\bar{\mu}}$ contracts exactly the same curves as $f'$ does, one has that $\tilde{f'}_{\bar{\mu}}\simeq f'$. 
We can write that  $$D=\psi^* \tilde{F}-\sum \tilde{E_i},$$ where $\tilde{E_i}$ (resp. $\tilde{F}$) is the lifting of $E_i$ (resp. $F$) to $S$ (resp. $T$) for all $i$.  
If $D$ is not of insufficiently fiber type, then $D \geq \epsilon \psi^* \tilde{F}$ for $1 \gg \epsilon >0$, which implies that $\kappa(D)=1$. 
This is a contradiction. 
Hence, $D$ has to be of insufficiently fiber type, 
but it then follows from Lemma \ref{lem-insufficient} again that $D=N$, that is, $P=0$. This is also a  contradiction. 
Therefore, we conclude that $-K_T$ is not semi-ample. 

Let $D'=\psi_*D \sim_{\Z} -K_T$. Then $D'_{\bar{\mu}}=F$. 
We have seen above that $D'$ is nef but not semi-ample, and $T_{\bar{\mu}}$ is globally $F$-split. 
By an argument very similar to the proof of Claim \ref{claim1}, one has that $\kappa(D')=0$ and $\nu(D')=1$. 

\begin{cl}\label{claim3} $(T, D')$ is log canonical.
\end{cl}
\begin{proof}[{Proof of Claim \ref{claim3}}]
Suppose to the contrary that the pair $(T,D')$ is not log canonical. 
Then $D'$ does not have normal crossing support or the coefficient of  a component of $D'$ is greater than one. 
Note that $D'_{\bar{\mu}}$ is an indecomposable curve of canonical type by Claim \ref{claim2}. 
We see from the condition on $D'$ that $D'_{\bar{\mu}}$ is not of type $I_n$. 

Since $H^1(T, -D')=H^1(T, \mathcal{O}_{T})=0$, the restriction map 
$$H^0(T, D') \to H^0(D', D'|_{D'})$$
is surjective. 
Note that $h^0(T, D')=1$, because $\kappa(D')=0$. 
By taking into account that $D'$ is not trivial, the surjectivity of the restriction map forces $H^0(D', D')$ to be zero. 
Thus, $\mathcal{O}_{T_{\bar{\mu}}} (D'_{\bar{\mu}})|_{D'_{\bar{\mu}}}$ is not trivial.  

By Lemma\,\ref{pic0}, $\mathcal{O}_{T_{\bar{\mu}}} (D'_{\bar{\mu}})|_{D'_{\bar{\mu}}}$ is a non-trivial $p(\mu)$-torsion element of $\mathrm{Pic}\,D'_{\bar{\mu}}$, where $p(\mu)$ is the characteristic of the field $k(\mu)$. 
Hence, $H^0(D'_{\bar{\mu}}, r D'_{\bar{\mu}})=0$ for $1\leq r \leq p(\mu)-1$. 
Using this inductively, one has $h^0(T_{\bar{\mu}},  r D'_{\bar{\mu}})=1$ for $1\leq r \leq p(\mu)-1$. 
In particular, 
\begin{equation}\tag{$\star$}
h^0(T_{\bar{\mu}}, (1-p(\mu)) K_{T_{\bar{\mu}}})=h^0(T_{\bar{\mu}}, (p(\mu)-1) D'_{\bar{\mu}})=1.
\end{equation}
The global $F$-splitting of $T_{\bar{\mu}}$ gives a section of 
$$\mathrm{Hom}(F_*\mathcal{O}_{T_{\bar{\mu}}}, \mathcal{O}_{T_{\bar{\mu}}}) \cong H^0(T_{\bar{\mu}}, (1-p(\mu)) K_{T_{\bar{\mu}}}),$$
and by $(\star)$,  the corresponding divisor has to be $(p(\mu)-1)D'_{\bar{\mu}}$. 
It then follows from an argument similar to the proof of \cite[Theorem 3.3]{hw} (see also the proof of  \cite[Theorem 4.3]{schsmith-logfano}) that $(T_{\bar{\mu}},D'_{\bar{\mu}})$ is globally sharply $F$-split. 
Since $\mu$ is any element of $W$, this means that $(T, D')$ is of globally sharply $F$-split type, so in particular, it is log canonical by Proposition \ref{just singularities} (1). This is a contradiction. 
\end{proof}

The log canonicity of $(T, D')$ implies the log canonicity of $(S, D)$, because $K_S+D=\psi^*(K_T+D')$. 
This is a contradiction and we finish the proof of Theorem \ref{F-split surface}. 
 \end{proof}

\begin{rem}\label{rem-fk}
Let $S$ be the blow-ups of $\mathbb{P}^2$ at $9$ general points $x_1, \dots, x_9 $ on a cuspidal  cubic curve $C$ in $\mathbb{P}^2$ (we work over the complex number field $\mathbb{C}$). 
Then 
$$\mathrm{Pic^0}C \simeq \mathbb{G}_a(\mathbb{C}),$$
so $-K_S$ is nef and $\kappa(-K_S)=0$. 
However, $S$ is not of Calabi--Yau type, because the log canonical threshold of $C$ is $5/6$. This contradicts the Main Theorem in \cite{fk}.
\end{rem}

\end{document}